\documentclass[12pt, reqno]{amsart}

\usepackage[mathscr]{eucal}
\usepackage[all]{xy}
\usepackage{mathrsfs}
\usepackage{xypic}
\usepackage{amsfonts}
\usepackage{amsmath}
\usepackage{amsthm}
\usepackage{amssymb}
\usepackage{eufrak}

\allowdisplaybreaks 

\newtheorem{theorem}{Theorem}
\newtheorem{lemma}[theorem]{Lemma}

\newtheorem{remark}[theorem]{Remark}

\theoremstyle{definition}
\newtheorem{definition}[theorem]{Definition}
\newtheorem{definition-lemma}[theorem]{Definition-Lemma}

\def\a{\alpha}
\def\b{\beta}

\def\k{\Bbbk}

\newcommand{\vi}{$\;\,{\sf {(i)}}\;$}
\newcommand{\vii}{$\;{\sf {(ii)}}\;$}
\newcommand{\viii}{${\sf {(iii)}}\;$}

\newcommand{\eps}{{\varepsilon}}

\newcommand{\ot}{{\otimes}}
\newcommand{\Q}{\mathbb{Q}}
\newcommand{\Z}{\mathbb{Z}}

\begin{document}

\title[Quantization of the Lie bialgebra of String Topology]{Quantization of the Lie Bialgebra\\ of String Topology}

\author{Xiaojun Chen, Farkhod Eshmatov, Wee Liang Gan}
\address{Xiaojun Chen and Farkhod Eshmatov\\ Department of Mathematics\\ University of Michigan\\ Ann Arbor, MI 48109, U.S.A.}
\email{xch@umich.edu, eshmatov@umich.edu}

\address{Wee Liang Gan\\ Department of Mathematics\\ University of California\\ Riverside, CA 92521, U.S.A.}
\email{wlgan@math.ucr.edu}


\maketitle

\begin{abstract}
Let $M$ be a smooth, simply-connected, closed oriented manifold,
and $LM$ the free loop space of $M$.
Using a Poincare duality model for $M$, we show that the reduced equivariant homology of $LM$ has the structure of a Lie bialgebra, and we construct a Hopf algebra which quantizes the Lie bialgebra.
\end{abstract}


\section{Introduction}

Let $M$ be a smooth, simply-connected, closed oriented manifold, and $LM$ the free loop space of 
$M$.
In this paper, by the \emph{reduced equivariant homology} of $LM$, we mean the $S^1$-equivariant 
homology of $LM$ relative to a point which is
a constant loop in $M$. We give an 
algebraic chain model of a Lie bialgebra structure 
on the reduced equivariant homology of the free loop space $LM$, and we construct a quantization of the Lie bialgebra.  Chas and Sullivan have shown in \cite{CS02}
that the $S^1$-equivariant  homology of $LM$ relative to the constant loops has the natural structure of a Lie bialgebra.  Our constructions are compatible with the one which they gave.

The motivations for the constructions in this paper are as follows.
In \cite{Goldman}, Goldman showed that the free homotopy
classes of closed curves on a Riemann surface $\mathcal S$
form a Lie algebra; and subsequently in \cite{Turaev}, Turaev proved that
they in fact form a Lie bialgebra. Moreover, Turaev showed
that the Lie bialgebra admits a quantization, which can be realized as
the skein algebra of links in $\mathcal S \times I$, where $I$ is the
unit interval. In  \cite{CS99} and \cite {CS02}, Chas and
Sullivan generalized the Lie bialgebra of Goldman and Turaev to the higher
dimensional case, and initiated the subject of string topology.
Algebraic chain models for string topology have been studied by many authors. 
In particular, Felix and Thomas \cite{FT}
used a Poincare duality model of the manifold $M$ that
was constructed by Lambrechts and Stanley \cite{LS07} to give a chain model of
the BV-algebra structure on the singular homology of $LM$ (see also \cite{HL}).

On the other hand, there is a quiver
analogue of the Lie bialgebra of Goldman and Turaev. This is the necklace Lie bialgebra 
studied by V. Ginzburg \cite{Ginzburg} and T. Schedler \cite{Schedler}. 
Schedler constructed a quantization of the  
necklace Lie bialgebra in analogy with Turaev's quantization. 
In this paper, we show that the Poincare duality model of $M$ can also be used to 
construct a chain model of a Lie bialgebra structure on the reduced equivariant homology of $LM$.
Furthermore, a quantization of the Lie bialgebra can be constructed following Schedler's method.

We will prove the following theorem
 (cf. Theorems~\ref{thm_of_Jones}, \ref{thm_Liebialg} and \ref{thm_quantization} of this paper):

\begin{theorem}\label{main_thm}
\vi Let $V$ be a counital coaugmented 
simply-connected DG open Frobenius algebra  over  a field $\k$ of
characteristic zero. 
 Then there is a natural involutive Lie bialgebra structure  on the reduced cyclic homology of $V$,
  and a Hopf algebra over $\k[h]$ (where $h$
is a formal parameter) which quantizes this Lie bialgebra.

\vii Given a smooth, closed, oriented and simply connected manifold $M$, there is a
counital coaugmented simply-connected DG open Frobenius
algebra $V$ over $\Q$ which models the chain complex of $M$, and the reduced cyclic homology
of $V$ is isomorphic to the reduced equivariant homology of $LM$. 
Therefore, the reduced equivariant homology of $LM$ has a Lie bialgebra structure, and
the Hopf algebra of \vi quantizes this Lie bialgebra.
\end{theorem}
The definition of a DG  open Frobenius algebra will be
given later  in Definition \ref{def_Frob}. By  cyclic homology, we mean the
homology of the Connes' complex for a coalgebra (see Definition \ref{hoch_cyc}). 
It would be interesting to find a geometric interpretation of the Hopf algebra in
Theorem  \ref{main_thm}, similar to Turaev's construction.
The construction in Theorem \ref{main_thm}\vii
of the DG open Frobenius algebra of a manifold is due to Lambrechts and
Stanley (see \cite{LS07}).

In \cite{CFP}, Cattaneo, Fr\"olich and Pedrini gave a topological
field theoretic interpretation
of string topology. In particular, they showed that under some mild
conditions of the gauge group, the Poisson bracket of the
generalized Wilson loops, when applied to two equivariant homology
classes in $LM$, is the same as the generalized Wilson loop applied
to the Lie bracket of two equivariant homology classes. They also
suggested several approaches of quantizing this field theory (see
\cite{CFP} \S8.3). We expect that the Hopf algebra given by Theorem \ref{main_thm}
 will appear in the associated quantum field
theory, which they conjectured to be related to algebraic structures on
the Vassiliev homology of links in higher dimensional manifolds.

Another  interesting problem is the existence of natural representations
   of the string topology Lie bialgebra and its quantization.
In the Riemann surface case, Goldman (\cite{Goldman}) showed that there is a Lie algebra homomorphism from the homotopy classes
of curves to the smooth functions on the moduli space of flat connections on the surface. 
 Similar results in the quiver case have been obtained by
Ginzburg (\cite{Ginzburg}) and Schedler (\cite{Schedler}). 
In the string topology case, a Lie algebra homomorphism from the equivariant homology to the smooth
functions on some related moduli space has been studied by Abbaspour, Zeinalian, and Tradler 
(see \cite{ATZ} and \cite{AZ}). Besides
 Turaev, the quantization of  the Lie algebra of curves has also been studied by
Andersen, Mattes and Reshetikhin (\cite{AMR1}, \cite{AMR2}) from the Vassiliev knot theory point of view.
We hope to address this problem in the future and relate it to
the work of Cattaneo and Rossi (see \cite{CR} and references therein).

The rest of the paper is devoted to the proof of
Theorem~\ref{main_thm}. In Section~\ref{homology_model} we recall
a chain complex model for the free loop space of manifolds, due independently to
K.-T. Chen \cite{Chen77}
and J.D.S. Jones \cite{Jo87}. 
From this, one can construct a chain model for the reduced
equivariant homology of the free loop space.
In Section~\ref{Liebialg} we construct the Lie
bialgebra structure on the reduced equivariant homology of the free loop
space, and in the last section, Section~\ref{quantization}, we
construct the Hopf algebra which quantizes the Lie bialgebra, following the work of Schedler. Our main computation is contained in the proof of Lemma \ref{subcomplex}.

The first author would like to thank Professor Yongbin Ruan for his encouragement during the preparation of this paper.

\section{Reduced equivariant homology}\label{homology_model}

\subsection{Cyclic homology of coalgebras}
First, let us recall the definitions of Hochschild homology and  cyclic homology of a coalgebra.

Let $(C, d)$ be a  DG coalgebra over a field $\k$ of characteristic 0. 
For any element $a\in C$, denote the coproduct of $a$ by $\sum_{(a)} a'\otimes a''$.
We shall write $C[1]$ for $C$ with degrees of elements shifted down by 1.

\begin{definition}\label{hoch_cyc}

\vi The {\it Hochschild
chain complex} of $C$, denoted by $(\mathrm{Hoch}_*(C), b)$, is the vector space 
$\prod_{n=0}^\infty C\otimes C[1]^{\otimes n}$
with the differential $b$ defined on homogeneous elements by:
\begin{eqnarray}&& b(a_0,a_1,\cdots,a_n)\label{twisted_diff}\\
&:=&
-\sum_{i=0}^n(-1)^{\varepsilon_{i-1}}(a_0,\cdots,da_i,\cdots,a_n)\label{internal}\\
&&+\sum_{i=0}^n \sum_{(a_i)} (-1)^{\varepsilon_{i-1}+|a_i'|-1}(a_0,\cdots,a_i',a_i'',\cdots,a_n)\label{external1}\\
&&+\sum_{(a_0)} (-1)^{(|a_0'|-1)(\varepsilon_n-|a_0'|)}(a_0'',a_1,\cdots,a_n,a_0'),\label{external2}
\end{eqnarray}
where $\varepsilon_i=|a_0|+\cdots+|a_i|-i$. The associated homology
is called the {\it Hochschild homology}, and is denoted by
$\mathrm{HH}_*(C)$.

\vii For any
$\a=(a_0,\cdots,a_n)\in\mathrm{Hoch}_*(C)$, define
$$t(\a):=(-1)^{(|a_0|-1)(\varepsilon_{n}-|a_0|)}
(a_1,\cdots,a_{n}, a_0),$$
and let $N=id+t+\cdots+t^{n}$. The image of $N$, denoted by
$\mathrm{CC}_*(C)$, is a subcomplex of $\mathrm{Hoch}_*(C)$, called the
{\it Connes  complex}. The homology of $\mathrm{CC}_*(C)$ is called the {\it
cyclic homology} of $C$, and is denoted by $\mathrm{HC}_*(C)$.
\end{definition}
In the above definition of a Hochschild chain complex, we call
(\ref{internal}) the internal differential, and
(\ref{external1})+(\ref{external2}) the external differential.
The Hochschild homology of a coalgebra is also known as the coHochschild homology (see \cite{HPS}).

\subsection{Poincare duality model} \label{pdm}
Throughout this paper, we assume that $M$ is a simply-connected, compact, oriented smooth manifold
of dimension $\mathsf m$.
Denote by $LM$ the free loop space of $M$. There is a natural $S^1$-action on $LM$.
We shall take $\k$ to be the field of rational numbers,
and write $C_*(-)$ and $C^*(-)$ for the singular chain complex and the singular cochain complex, respectively. We grade $C^*(-)$ negatively. 

The chain complex $C_*(M)$ has a partially defined
product given by intersection of {\it transversal} chains, and
a coproduct given by the Alexander-Whitney approximation of the diagonal embedding
$M\hookrightarrow M\times M$. In fact, $C_*(M)$ is a partially defined (non-commutative) DG open Frobenius algebra over $\Z$; however, over $\Q$, one may define
the Frobenius algebra structure fully, due to 
a result by P. Lambrechts and D. Stanley \cite{LS07}:

\begin{definition}[Open Frobenius algebra]\label{def_Frob}
Let $V$ be a DG vector space over $\k$. A DG {\it open Frobenius algebra}
(of degree $\mathsf m$) on $V$
is the triple $(V,\cdot, \Delta)$ such that:

\vi $(V,\cdot)$ is a 
commutative DG algebra (whose product is of degree $-\mathsf m$); 

\vii $(V,\Delta)$ is a cocommutative DG coalgebra;
 
\viii The following
identity, called the {\it module compatibility}, holds: for any
$a,b\in V$,
\begin{equation}\label{Frobenius_id}\Delta(a\cdot b)
=\sum (-1)^{\mathsf m |a'|} a'\otimes a''\cdot b
= \sum a\cdot b'\otimes b'',\end{equation} 
where $\Delta a=\sum a'\otimes
a''$, and $\Delta b=\sum a'\otimes b''$. 
\end{definition}

\begin{theorem}[Lambrechts and Stanley]\label{poincare} 
There is a finite dimensional commutative DG algebra $A$ such that
$A$ is simply-connected, $A$ is quasi-isomorphic to $C^*(M)$, and there is an $A$-bimodule
isomorphism of degree $\mathsf m$ from $A$ to its dual $A^\vee$ that induces the Poincare duality isomorphism $H^*(M)\to H_{*+\mathsf m}(M)$ on homology.
\end{theorem}
\begin{proof}
This is immediate from applying \cite[Theorem 1.1]{LS07} to the Sullivan minimal model of $M$.
\end{proof}

The DG algebra $A$ in Theorem \ref{poincare} is called a Poincare duality model for $M$.
For the rest of this section, we let $V=A^\vee$, where $A$ is a Poincare duality model.
Thus, $V$ is a   cocommutative DG coalgebra. The linear isomorphism
from $A$ to $V[\mathsf m]$ induces the structure of a 
commutative DG algebra on $V$ whose product has degree $-\mathsf m$. Moreover, the coproduct is a morphism of $V$-bimodules. Therefore, $V$ is a DG open Frobenius algebra which is simply-connected, 
has a counit $\eps: V\to \k$, and a coaugmentation 
$\eta: \k\hookrightarrow V$.
Let $C=\mathrm{Coker}(\eta)$, the coaugmentation coideal of $V$. 
The homology of $\mathrm{CC}_*(C)$ is called the \emph{reduced cyclic homology} of $V$.

The following theorem is essentially due to  Jones \cite{Jo87}. 

\begin{theorem}\label{thm_of_Jones}
The reduced equivariant chain complex of $LM$  is quasi-isomorphic to
$\mathrm{CC}_*(C)[1]$.
\end{theorem}

\begin{proof}
The Hochschild complex $\mathrm{Hoch}_*(V)$ is quasi-isomorphic to the normalized Hochschild complex $\prod_{n=0}^\infty V\otimes C[1]^{\otimes n}$ whose differential $b$ is given by the same formula (\ref{twisted_diff}). 
For the rest of this proof, we shall denote by
$\mathrm{Hoch}_*(V)$ 
 the normalized Hochschild complex of $V$. Then Connes' operator $B$ is defined by 
$$\begin{array}{cccl}
B:&\mathrm{Hoch}_*(V)&\longrightarrow&\mathrm{Hoch}_{*+1}(V)\\
&(a_0,a_1,\cdots,a_n)&\longmapsto&\displaystyle\sum_{i=1}^n\varepsilon(a_0)(a_i,\cdots,a_n,
a_1,\cdots, a_{i-1}),\end{array}$$
One has $B^2=0$, $b\circ B+B\circ b=0$.
Let $u$ be a formal variable of degree 2, and define the differential
$$b+u^{-1}B: \mathrm{Hoch}_*(V)[u]\to \mathrm{Hoch}_*(V)[u]$$ by
$$(b+u^{-1}B)(\a\otimes u^n)=\left\{\begin{array}{ll} b(\a)\otimes u^n+B(\a)\otimes u^{n-1},&\mbox{if}\quad n>0,\\
b(\a),&\mbox{if}\quad n=0,\end{array}\right.$$

Since $A$ is quasi-isomorphic to $C^*(M)$, it follows by a well-known result
of Jones (see \cite[Theorem A]{Jo87} or \cite[Theorem 1.5.1, Corollary 1.5.2]{CV})
that the equivariant chain complex
$C_*^{S^1}(LM)$ is quasi-isomorphic to
$(\mathrm{Hoch}_*(V)[u], b+u^{-1}B)$. Let $$\overline{\mathrm{Hoch}}_*(V)=
C\oplus \prod_{n=1}^\infty V\otimes C[1]^{\otimes n}$$
be the reduced Hochschild complex. Then the reduced equivariant chain complex of 
$LM$
and $(\overline{\mathrm{Hoch}}_*(V)[u], b+u^{-1}B)$ are quasi-isomorphic.
On the other hand, note that  the normalized 
and reduced Hochschild complexes are non-negatively graded, so the
bicomplex $(\overline{\mathrm{Hoch}}_*(V)[u], b+u^{-1}B)$ lies in 
 the first quadrant.
It follows by a standard argument using filtrations (see 
\cite[Proposition 2.2.14]{Loday}) 
that the map
\begin{gather*}
(\overline{\mathrm{Hoch}}_*(V)[u], b+u^{-1}B) \to CC_*(C)[1] \\
\a\otimes u^n \mapsto
\left\{\begin{array}{ll} B(\a) ,&\mbox{if}\quad n=0, \\
0,&\mbox{if}\quad n>0,\end{array}\right.
\end{gather*}
is a quasi-isomorphism. The gives the desired result.
\end{proof}

\begin{remark}
The equivariant homology of $LM$ is isomorphic to the direct sum of the reduced equivariant homology of $LM$ with $\k[u]$ (where $\deg(u)=2$).
\end{remark}

Since $V$ is simply-connected, we may replace the direct product in the
definition of $\mathrm{Hoch}_*(C)$ by direct sum.

\section{Lie bialgebra}\label{Liebialg}

\subsection{Construction of the Lie bialgebra}\label{def_L}

In this section, $C$ is the coaugmentation coideal of  a 
counital coaugmented simply-connected DG open Frobenius algebra $V$ (of degree $\mathsf m$). 
One can also take $C$ to be $V$ itself.
We shall write $\pm$ for signs determined by the usual Koszul convention.

\begin{definition}[Lie coalgebra] 
Let $L$ be a vector space over $\k$.
 A skew-symmetric map $\delta:L\to L\otimes L$
defines a Lie coalgebra structure on $L$ if
\begin{equation}\label{co-Jacobi}(\tau^2+\tau+id)\circ(\delta\otimes id)\circ\delta=0:L\to L\otimes L\otimes L,\end{equation}
where $\tau$ is the permutation $a\otimes b\otimes c\mapsto\pm
c\otimes a\otimes b$, for $a$, $b$, $c\in L$. The map $\delta$ is
called the {\it cobracket} and (\ref{co-Jacobi}) is called the {\it
co-Jacobi identity}.
\end{definition}

\begin{definition}[Lie bialgebra] Suppose $(L,\{\,,\,\})$ is a Lie
algebra and $(L, \delta)$ is a Lie coalgebra. The triple $(L,\{\,,\,\},\delta)$ defines a Lie bialgebra on $L$ if the following identity,
called the Drinfeld compatibility, holds:
\begin{equation}\label{Drinfeld_compatibility}\delta\{a,b\}=\{id\otimes a+a\otimes id,\delta(b)\}+\{\delta(a),id\otimes b+b\otimes id\}.\end{equation}
If moreover, $\{\,,\,\}\circ\delta:L\to L$ vanishes identically, the Lie bialgebra $(L,\{\,,\,\},\delta)$ is called {\it involutive}.
\end{definition}

Let $L:=\Big(\mathrm{CC}_*(C)[1]\Big)[\mathsf m-2]
= \mathrm{CC}_*(C)[\mathsf m-1]$, 
where $\mathrm{CC}_*(C)$ is defined in the
previous section.
We shall write elements of
$L$ in the form $N([a_1|\cdots|a_n])$, where
$a_i\in C[1]$ for $i=1,\cdots, n$.

Define on $L$ the following two operators:

$\{\,,\,\}:L\otimes L\to L$ by
\begin{equation}\label{Lie_bracket}\{\a,\b\}:=\sum_{i,j}\pm
\varepsilon(a_i\cdot
b_j)
N([a_{i+1}|\cdots|a_n|a_{1}|\cdots|a_{i-1}|b_{j+1}|\cdots|b_m|b_1|\cdots|b_{j-1}])\end{equation}
and $\delta:L\to L\otimes L$ by
\begin{equation}\label{wedge_prod}\delta(\a):=\sum_{i<j}\pm
\varepsilon(a_i\cdot
a_j) N([a_1|\cdots|a_{i-1}|a_{j+1}|\cdots|a_n])\wedge
N([a_{i+1}|\cdots|a_{j-1}]),\end{equation} for any homogeneous
$\a=N([a_1|\cdots|a_n]),\b=N([b_1|\cdots|b_m])\in L$, where in above
$\varepsilon$ is the coaugmentation, and in (\ref{wedge_prod}),
$a\wedge b$ means $a\otimes b-b\otimes a$, and will also be
written as $a\otimes b-\mathrm{Alt}$.

\begin{theorem}\label{thm_Liebialg}Let $L$ be as above. Then $(L,\{\,,\,\},\delta)$ forms an involutive DG Lie bialgebra.\end{theorem}

The Lie bracket $\{\,,\,\}$ is of degree $0$, 
 and the Lie cobracket $\delta$ is of 
degree $2(2-\mathsf m)$.
The rest of this section is devoted to the proof of Theorem \ref{thm_Liebialg}.
The proof is divided into several steps.

\subsection{Proof of the DG Lie algebra} 
The  product on $V$
is graded commutative, hence if we shift the degree of $C$ down by
1, the induced pairing $\varepsilon(a\cdot b): C[1]\otimes
C[1]\to \k$ is graded skew-symmetric. Therefore the bracket
$\{\,,\,\}$ defined by (\ref{Lie_bracket}) is graded skew-symmetric.
We now show the Jacobi identity: for any $\a=N([a_1|\cdots|a_n]),
\b=N([b_1|\cdots|b_m]),\gamma=N([c_1|\cdots|c_p])\in L$,
\begin{eqnarray}
&&\{\{\a,\b\},\gamma\}\nonumber\\
&=&\sum_{i,j,k,l}\pm \varepsilon(a_ib_j)\varepsilon(a_kc_l)N([a_1|\cdots|b_{j+1}|\cdots|b_{j-1}|\cdots|c_{l+1}|\cdots|c_{l-1}|\cdots|a_n])\label{jacobi1}\\
&+&\sum_{i,j,k,l}\pm \varepsilon(a_ib_j)\varepsilon(b_kc_l)N([a_1|\cdots|b_{j+1}|\cdots|c_{l+1}|\cdots|c_{l-1}|\cdots|b_{j-1}|\cdots|a_n]),\label{jacobi2}
\end{eqnarray}
Similarly, we have
\begin{eqnarray}
&&\{\{\b,\gamma\},\a\}\nonumber\\
&=&\sum_{i,j,k,l}\pm \varepsilon(b_jc_l)\varepsilon(b_ka_i)N([b_1|\cdots|c_{l+1}|\cdots|c_{l-1}|\cdots|a_{i+1}|\cdots|a_{i-1}|\cdots|b_m])\label{jacobi3}\\
&+&\sum_{i,j,k,l}\pm \varepsilon(b_jc_l)\varepsilon(c_ka_i)N([b_1|\cdots|c_{l+1}|\cdots|a_{i+1}|\cdots|a_{i-1}|\cdots|c_{l-1}|\cdots|b_m]),\label{jacobi4}
\end{eqnarray}
and
\begin{eqnarray}
&&\{\{\a,\b\},\gamma\}\nonumber\\
&=&\sum_{i,j,k,l}\pm \varepsilon(c_la_i)\varepsilon(c_kb_j)N([c_1|\cdots|a_{i+1}|\cdots|a_{i-1}|\cdots|b_{j+1}|\cdots|b_{j-1}|\cdots|c_p])\label{jacobi5}\\
&+&\sum_{i,j,k,l}\pm \varepsilon(c_la_i)\varepsilon(a_kb_j)N([c_1|\cdots|a_{i+1}|\cdots|b_{j+1}|\cdots|b_{j-1}|\cdots|a_{i-1}|\cdots|c_p]).\label{jacobi6}
\end{eqnarray}
Note that by the cyclic invariance of $N$, (\ref{jacobi1}) cancels with (\ref{jacobi6}),
so do (\ref{jacobi2}) with (\ref{jacobi3}) and (\ref{jacobi4}) with (\ref{jacobi5}).
This proves the Jacobi identity.

We next show that $b$ respects the bracket. It is easy to see that the bracket thus defined commutes with the internal differential, hence we only check that it commutes with the external differential.
For any $\a=N([a_1|\cdots|a_n]), \b=N([b_1|\cdots|b_m])$,
$$b(\a)=\sum_{i=1}^n N([a_1|\cdots|a_i'|a_i''|\cdots|a_n]),\quad\mbox{and}\quad b(\b)=\sum_{j=1}^n N([b_1|\cdots|b_j'|b_j''|\cdots|b_m]).$$
Therefore,
\begin{eqnarray}
&&\{b(\a),\b\}\nonumber\\
&=& \sum_{i,k,l}\pm\varepsilon(a_kb_l)N([a_{k+1}|\cdots|a_n|a_1|\cdots|a_i'|a_i''|\cdots|a_{k-1}|b_{l+1}|\cdots|b_m|b_1|\cdots|b_{l-1}])\label{Part_1}\\
&+&\sum_{i,l}\pm\varepsilon(a_i'b_l)N([a_i''|\cdots|a_n|a_1|\cdots|a_{i-1}|b_{l+1}|\cdots|b_m|b_1|\cdots|b_{l-1}])\label{Part_2}\\
&+&\sum_{i,l}\pm\varepsilon(a_i''b_l)N([a_{i+1}|\cdots|a_n|a_1|\cdots|a_{i-1}|a_i'|b_{l+1}|\cdots|b_m|b_1|\cdots|b_{l-1}]),\label{Part_3}\\
&&\{\a,b(\b)\}\nonumber\\
&=&\sum_{j,k,l}\pm\varepsilon(a_kb_l)N([a_{k+1}|\cdots|a_n|a_1|\cdots|a_{k-1}|b_{l+1}|\cdots|b_j'|b_j''|\cdots|b_m|b_1|\cdots|b_{l-1}])\label{Part_4}\\
&+&\sum_{k,j}\pm\varepsilon(a_kb_j')N([a_{k+1}|\cdots|a_n|a_1|\cdots|a_{k-1}|b_j''|\cdots|b_m|b_1|\cdots|b_{j-1}])\label{Part_5}\\
&+&\sum_{k,j}\pm\varepsilon(a_kb_j'')N([a_{k+1}|\cdots|a_n|a_1|\cdots|a_{k-1}|b_{j+1}|\cdots|b_m|b_1|\cdots|b_{j-1}|b_j']),\label{Part_6}
\end{eqnarray}
while
\begin{eqnarray}
&&b\{\a,\b\}\nonumber\\
&=& \sum_{i,k,l}\pm\varepsilon(a_kb_l)N([a_{k+1}|\cdots|a_n|a_1|\cdots|a_i'|a_i''|\cdots|a_{k-1}|b_{l+1}|\cdots|b_m|b_1|\cdots|b_{l-1}])\label{Part_7}\\
&+&\sum_{j,k,l}\pm\varepsilon(a_kb_l)N([a_{k+1}|\cdots|a_n|a_1|\cdots|a_{k-1}|b_{l+1}|\cdots|b_j'|b_j''|\cdots|b_m|b_1|\cdots|b_{l-1}]).\label{Part_8}
\end{eqnarray}
Note that (\ref{Part_2}) and (\ref{Part_6}) cancel, so do
(\ref{Part_3}) and (\ref{Part_5}). The remaining terms of
$\{b(\a),\b\}+\{\a,b(\b)\}$ are identical to
$(\ref{Part_7})+(\ref{Part_8})$, which is exactly $b\{\a,\b\}$.

\subsection{Proof of the DG Lie coalgebra} 
The cobracket is
skew-symmetric. The co-Jacobi identity holds due by a similar computation as the Jacobi identity, and so we leave its verification to the reader.

Next, we show that $b$ respects the cobracket. As before, we check 
that the external differential commutes with the cobracket: by definition,
\begin{eqnarray}
&&\delta N([a_1|\cdots|a_n])\nonumber\\
&=&\sum_{i<j}\pm\varepsilon(a_ia_j)N([a_1|\cdots|a_{i-1}|a_{j+1}|\cdots|a_n])\otimes
N([a_{i+1}|\cdots|a_{j-1}])\nonumber-\mathrm{Alt},
\end{eqnarray}
hence
\begin{eqnarray}
&&b(\delta N([a_1|\cdots|a_n])\nonumber\\
&=&\sum_{i<j, k}\pm\varepsilon(a_ia_j)N([a_1|\cdots|a_k'|a_k''|\cdots|a_{i-1}|a_{j+1}|\cdots|a_n])\otimes N([a_{i+1}|\cdots|a_{j-1}])-\mathrm{Alt}\quad\label{Part_9}\\
&+&\sum_{i<j,
l}\pm\varepsilon(a_ia_j)N([a_1|\cdots|a_{i-1}|a_{j+1}|\cdots|a_n])\otimes
N([a_{i+1}|\cdots|a_l'|a_l''|\cdots|a_{j-1}])-\mathrm{Alt},\quad\label{Part_10}
\end{eqnarray}
while $\delta \sum_k N([a_1|\cdots|a_k'|a_k''|\cdots|a_n])$ has not only (\ref{Part_9}) and (\ref{Part_10}), but also
\begin{eqnarray}
&&\sum_{i,k}\pm\varepsilon(a_ia_k')N(a_1|\cdots|a_{i-1}|a_k''|\cdots|a_n])\otimes N([a_{i+1}|\cdots|a_{k-1}])-\mathrm{Alt}\label{Part_11}\\
&+&\sum_{i,k}\pm\varepsilon(a_ia_k'')N([a_1|\cdots|a_{i-1}|a_{k+1}|\cdots|a_n])\otimes N([a_{i+1}|\cdots|a_{k-1}|a_k'])-\mathrm{Alt}\quad\label{Part_12}\\
&+&\sum_{k,j}\pm\varepsilon(a_i'a_j)N([a_1|\cdots|a_{k-1}|a_{j+1}|\cdots|a_n])\otimes N([a_{k}''|\a_{k+1}|\cdots|a_{j-1}])-\mathrm{Alt}\quad\label{Part_13}\\
&+&\sum_{k,j}\pm\varepsilon(a_k''a_j)N([a_1|\cdots|a_k'|a_{j+1}|\cdots|a_n])\otimes
N([a_{k+1}|\cdots|a_{j-1}])-\mathrm{Alt}.\label{Part_14}
\end{eqnarray}
Since $C$ is a open Frobenius algebra, by the module compatibility
(\ref{Frobenius_id}), (\ref{Part_11}) cancels with (\ref{Part_14}),
and (\ref{Part_12}) cancels with (\ref{Part_13}), and hence $b$
commutes with the cobracket.

\subsection{Proof of the Drinfeld compatibility}
Let $\a=N([a_1|\cdots|a_n])$ and $\b=N([b_1|\cdots|b_m])$, and write
$\delta(\a)=\a^{(1)} \ot \a^{(2)}$ and $\delta(\b)=\b^{(1)}\ot
\b^{(2)}$.

We have
$$\{\a,\b\}
= \sum_{i,j} \pm \eps(a_ib_j)
N([a_{i+1}|\cdots|a_{i-1}|b_{j+1}|\cdots|b_{j-1}] )$$ and
\begin{align*}
& \delta\{\a,\b\} \\
=&  \sum_{i,j,k,l} \pm \eps(a_ib_j)\eps(a_k a_l)
N([a_{k+1}|\cdots|a_{l-1}]) \ot N([a_{l+1}|\cdots|a_{i-1}|b_{j+1}|\cdots|b_{j-1}|a_{i+1}|\cdots|a_{k-1}] )\\
+&   \sum_{i,j,k,l} \pm \eps(a_ib_j)\eps(a_k b_l)
N([a_{k+1}|\cdots|a_{i-1}|b_{j+1}|\cdots|b_{l-1}]) \ot
N([b_{l+1}|\cdots|b_{j-1}|a_{i+1}|\cdots|a_{k-1}]) \\
+ & \sum_{i,j,k,l} \pm \eps(a_ib_j)\eps(a_l a_k)
N([a_{k+1}|\cdots|a_{i-1}|b_{j+1}|\cdots|b_{j-1}|a_{i+1}|\cdots|a_{l-1}])
\ot
N([a_{l+1}|\cdots|a_{k-1}]) \\
+ & \sum_{i,j,k,l} \pm \eps(a_ib_j)\eps(b_k b_l)
N([b_{k+1}|\cdots|b_{l-1}]) \ot
N([b_{l+1}|\cdots|b_{j-1}|a_{i+1}|\cdots|a_{i-1}|b_{j+1}|\cdots|b_{k-1}]) \\
+ & \sum_{i,j,k,l} \pm \eps(a_ib_j)\eps(b_k a_l)
N([b_{k+1}|\cdots|b_{j-1}|a_{i+1}|\cdots|a_{l-1}]) \ot
N([a_{l+1}|\cdots|a_{i-1}|b_{j+1}|\cdots|b_{k-1}]) \\
+ &\sum_{i,j,k,l} \pm \eps(a_ib_j)\eps(b_l b_k)
N([b_{k+1}|\cdots|b_{j-1}|a_{i+1}|\cdots|a_{i-1}|b_{j+1}|\cdots|b_{l-1}])
\ot N([b_{l+1}|\cdots|b_{k-1}])
\end{align*}

In above, the second summation and the fifth summation cancel with
each other. The first summation is equal to $\a^{(1)} \ot
\{\a^{(2)}, \b\}$; the third summation is equal to $\{\a^{(1)}, \b\}
\ot \a^{(2)}$; the forth summation is equal to $\b^{(1)} \ot \{\a,
\b^{(2)}\}$; and the sixth summation is equal to $\{\a, \b^{(1)}\}
\ot \b^{(2)}$. Thus we obtain the Drinfeld compatibility.

\subsection{Proof of the involutivity}
 Let $\a=N([a_1|\cdots|a_n])$, then
\begin{eqnarray*}
\delta(\a)&=&\sum_{i<j}\pm\varepsilon(a_ia_j)N([a_1|\cdots|a_{i-1}|a_{j+1}|\cdots|a_n])\otimes N([a_{i+1}|\cdots|a_{j-1}])\\
&-&\sum_{i<j}\pm\varepsilon(a_ia_j)N([a_{i+1}|\cdots|a_{j-1}])\otimes N([a_1|\cdots|a_{i-1}|a_{j+1}|\cdots|a_n]).
\end{eqnarray*}
By a similar argument as above, one checks that
$\{\,,\,\}\circ\delta=0$ holds identically.

The above constructions and proofs, 
except for the compatibilities of the differential
with the Lie bracket and cobracket, are similar to the proof of the Lie
bialgebras of Turaev \cite{Turaev}, Chas-Sullivan \cite{CS02}, Hamilton \cite{Hamilton} and Schedler \cite{Schedler}.

\section{Quantization of the Lie bialgebra}\label{quantization}

\subsection{Construction of the Hopf algebra}
In this section, we construct a DG Hopf algebra which quantizes the DG Lie 
bialgebra of section \ref{Liebialg}. We follow Schedler \cite{Schedler} closely.
We will also define a new differential in Definition \ref{diffb} below which is 
not present in \cite{Schedler}.

\begin{definition}[Quantization] Let $h$ be a formal parameter.
Suppose $A$ is a Hopf algebra over $\k[h]$. We say $A$ quantizes
the Lie bialgebra $(L,\{\,,\,\},\delta)$ if there is a Hopf algebra
isomorphism
$$\phi:A/hA\stackrel{\cong}{\longrightarrow} U(L),$$
where $U(L)$ is the universal enveloping algebra of $L$, such that
for any $x_0\in L$, and any $x\in A$, $\phi(x)=x_0$,
$${1\over h}(\Delta(x)-\Delta^{\rm op}(x))\equiv \delta(x_0)\mod h,$$
where $\Delta^{\rm op}$ is the opposite comultiplication of
$A$.\end{definition}

\begin{definition}  \label{CH}
Let $CH:= C\otimes_\k \k[\mu, \mu^{-1}]$, 
where $\mu$ is a formal variable (of degree 0).
We shall write an element $a\otimes \mu^u \in CH$ as $(a,u)$ and
call  $u\in \mathbb Z$ the \emph{height} of $(a,u)$.
\end{definition}

Let $\widehat{\mathrm{Hoch}}_*(CH)$ be the graded vector space
$\bigoplus_{n=0}^\infty CH\otimes CH[1]^{\otimes n}$, and denote by
$\widehat{\mathrm{CC}}_*(CH)$ the subspace of cyclically invariant elements in 
$\widehat{\mathrm{Hoch}}_*(CH)$. 
Let $$LH:= \Big( \widehat{\mathrm{CC}}_*(CH) [1]\Big)[\mathsf m -2]
=\widehat{\mathrm{CC}}_*(CH) [\mathsf m-1] .$$
There is a canonical
projection $LH\to L$ by forgetting the heights in $LH$ (recall $L$
is given in \S\ref{def_L}).
Let $SLH$ be the symmetric algebra of $LH$.

\begin{definition} \label{diffb}
Define a differential $b$ on $SLH$ so that on the homogeneous components it is
given by
\begin{eqnarray}
&& b\Big( 
N([(a_{1,1},h_{1,1})|\cdots|
(a_{1,p_1},h_{1,p_1})])\bullet\cdots\bullet
N([(a_{n,1},h_{n,1})|\cdots|(a_{n,p_n},h_{n,p_n})])
\Big) \\
&:=&  
-  \sum_{i=1}^n  \sum_{j=1}^{p_i} 
\pm \cdots\bullet N([(a_{i,1},h_{i,1})|\cdots|(da_{i,j},h_{i,j})|\cdots|
(a_{i,p_i},h_{i,p_i})])\bullet \cdots \\
&&+ \sum_{i=1}^n
\sum_{j=1}^{p_i} \sum_{(a_{(i,j)})} \pm
N([(a_{1,1},\widehat h_{1,1})|\cdots| (a_{1,p_1},\widehat h_{1,p_1})])\bullet \cdots  \label{bh0} \\
&&  \cdots \bullet
N([(a_{i,1},\widehat h_{i,1})|\cdots|
(a_{i,j}',h_{i,j})| (a_{i,j}'',h_{i,j}+1)|
\cdots|(a_{i,p_i},\widehat h_{i,p_i})]) \bullet  \cdots \label{bh} 
\end{eqnarray}
where in (\ref{bh0}) and (\ref{bh}), for all $(i',j')\neq (i,j)$, 
\begin{equation}\label{hath}
\widehat h_{i',j'} = \left\{\begin{array}{ll}
  h_{i',j'} & \mbox{ if } h_{i',j'}\leq h_{i,j} \\
  h_{i',j'}+1 & \mbox{ if } h_{i',j'}>h_{i,j} .
\end{array}\right.
\end{equation}
\end{definition}

\begin{remark}In the above definition, we
assign $a_{i,j}'$ and $a_{i,j}''$ with heights $h_{i,j}$ and
$h_{i,j}+1$ respectively, and raise all heights $h_{i',j'}$ 
greater than $h_{i,j}$ to 
$h_{i,j}+1$. The coassociativity of $C$ implies that $b^2=0$.
\end{remark}

Let $h$ be a formal parameter of degree $2(\mathsf m-2)$.
The differential $b$ on $SLH$ extends to a differential on the $\k[h]$-module
$SLH[h]$.
Consider the subcomplex of $SLH[h]$ which is spanned by elements
whose homogeneous components are of the form:
\begin{equation}\label{typical_elmt}
N[(a_{1,1},h_{1,1})|\cdots|
(a_{1,p_1},h_{1,p_1})])\bullet\cdots\bullet
N([(a_{k,1},h_{k,1})|\cdots|(a_{k,p_k},h_{k,p_k})]),
\end{equation}
where all the $h_{i,j}$ are distinct. Denote this subcomplex by
$\widetilde{SLH}[h]$.

Let $\tilde A$ be the quotient module of $\widetilde{SLH}[h]$ defined by
identifying any element of the form (\ref{typical_elmt}) with other
elements obtained by replacing $h_{i,j}$ with any $\tilde h_{i,j}$ satisfying
$h_{i,j}<h_{i',j'}$ if and only if $\tilde h_{i,j}<\tilde h_{i',j'}$.  Let $\tilde B$ be the submodule of $\tilde A$
generated by elements of the following form:
\begin{equation}
\begin{array}{ll}X-X_{i,j,i',j'}'-X_{i,j,i',j'}'', &\mbox{
where }
i\ne i', h_{i,j}<h_{i',j'},\\
& \mbox{ and }\nexists (i'',j'') \mbox{ with }
h_{i,j}<h_{i'',j''}<h_{i',j'};\end{array}\label{equiv_rel1}\end{equation}
\begin{equation}
\begin{array}{ll}X-X_{i,j,i,j'}'-hX_{i,j,i,j'}'', &\mbox{ where }
h_{i,j}<h_{i,j'},\\
&\mbox{ and } \nexists (i'',j'') \mbox{ with
}h_{i,j}<h_{i'',j''}<h_{i,j'},
\end{array}\label{equiv_rel2}\end{equation}
where the $X'$ and $X''$ terms are defined as follows: if $i\ne i'$,
$X_{i,j,i',j'}'$ is the same as $X$ except that the heights
$h_{i,j}$ and $h_{i',j'}$ are interchanged, while $X_{i,j,i',j'}''$
replaces the components
$N([(a_{i,1},h_{i,1})|\cdots|(a_{i,p_i},h_{i,p_i})])$ and
$N([(a_{i',1},h_{i',1})|\cdots|(a_{i',p_{i'}},h_{i',p_{i'}})])$ by
$$\pm\varepsilon( a_{i,j}a_{i',j'})
N([(a_{i,j+1},h_{i,j+1})|\cdots|(a_{i,j-1},h_{i,j-1})|(a_{i',j'+1},h_{i',j'+1})|
\cdots|(a_{i',j'-1},h_{i',j'-1})]);$$ similarly, $X_{i,j,i,j'}'$ is
the same as $X$ but with the heights $h_{i,j}$ and $h_{i,j'}$
interchanged, while $X_{i,j,i,j'}''$ is given by replacing the
component with the following two components
$$\pm\varepsilon( a_{i,j}a_{i,j'}) N([(a_{i,j'+1},h_{i,j'+1})|\cdots|(a_{i,j-1},h_{i,j-1})])\bullet N([(a_{i,j+1},h_{i,j+1})|\cdots|(a_{i,j'-1},h_{i,j'-1})]).$$

\begin{lemma}\label{subcomplex}
Let $\tilde A$ and $\tilde B$ be as above. Then $\tilde A$ is a chain complex and $\tilde B$ is a subcomplex of $\tilde A$. \end{lemma} 

\begin{proof}
It is clear that $b$ is well-defined on $\tilde A$, and so $\tilde A$ is a 
chain complex.
We have to check that $\tilde B$ is a subcomplex of $\tilde A$.

The equivalence relation (\ref{equiv_rel1}) only involves operations on two components in the elements
of $\widetilde{SLH}[h]$, so without loss of generality, we may assume
$$X=N([(a_1,2h_1)|\cdots|(a_n,2h_n)])\bullet N([(b_1,2g_1)|\cdots|(b_m,2g_m)]).$$
Suppose that in $X$, the heights $2h_i$ and $2g_j$ satisfy condition (\ref{equiv_rel1}). Then
\begin{eqnarray*}
&&X-X'-X''\\
&=&\pm N([(a_1,2h_1)|\cdots|(a_n,2h_n)])\bullet N([(b_1,2g_1)|\cdots|(b_m,2g_m)])\\
&\mp&N([(a_1,2h_1)|\cdots|(a_i,2g_j)|\cdots|(a_n,2h_n)])\bullet N([(b_1,2g_1)|\cdots|(b_j,2h_i)|\cdots|(b_m,2g_m)])\\
&\mp&\varepsilon(a_ib_j) N([(a_{i+1},2h_{i+1})|\cdots|(a_{i-1},2h_{i-1})|(b_{j+1},2g_{j+1})|
\cdots|(b_{j-1},2g_{j-1})]).\end{eqnarray*}
Therefore
\begin{eqnarray}
b(X)&=&\sum_{k\ne i}\pm N([\cdots|(a_k',2h_k)|(a_k'',2h_k+1)|\cdots])\otimes N([\cdots|(b_l,2g_l)|\cdots])\label{X_1}\\
&+&\sum\pm N([\cdots|(a_i',2h_i)|(a_i'',2h_i+1)|\cdots])\otimes N([\cdots|(b_l,2g_l)|\cdots])\label{X_2}\\
&+&\sum_{l\ne j}\pm N([\cdots|(a_k,2h_k)|\cdots])\otimes N([\cdots|(b_l', 2g_l)|(b_l'',2g_l+1)|\cdots])\label{X_3}\\
&+&\sum\pm N([\cdots|(a_k,2h_k)|\cdots]\otimes N([\cdots|(b_j',2g_j)|(b_j'',2g_j+1)|\cdots]),\label{X_4}\\
b(X')
&=&\sum_{k\ne i}\pm N([\cdots|(a_k',2h_k)|(a_k'',2h_k+1)|\cdots])\otimes N([\cdots|(b_j,2h_i)|\cdots])\label{X'_1}\\
&+&\sum\pm N([\cdots|(a_i', 2g_j)|(a_i'',2g_j+1)|\cdots])\otimes N([(\cdots|(b_j,2h_i)|\cdots)])\label{X'_2}\\
&+&\sum_{l\ne j}\pm N([\cdots|(a_i,2g_j)|\cdots])\otimes N([\cdots|(b_l',2g_l)|(b_l'',2g_l+1)|\cdots])\label{X'_3}\\
&+&\sum\pm N([\cdots|(a_i,2g_j)|\cdots])\otimes N([\cdots|(b_j',2h_i)|(b_j'',2h_{i}+1)|\cdots]),\label{X'_4}\\
b(X'')&=&\sum_{k\ne i}\pm\varepsilon(a_ib_j) N([\cdots|(a_k', 2h_k)|(a_k'',2h_k+1)|\cdots])\label{X''_1}\\
&+&\sum_{l\ne j}\pm\varepsilon(a_ib_j)N([\cdots|(b_l',
2g_l)|(b_l'',2g_l+1)|\cdots]).\label{X''_2}
\end{eqnarray}

It is plain that both $(\ref{X_1})-(\ref{X'_1})-(\ref{X''_1})$ and
$(\ref{X_3})-(\ref{X'_3})-(\ref{X''_2})$ are contained in $\tilde B$. To see that
$(\ref{X_2})+(\ref{X_4})-(\ref{X'_2})-(\ref{X'_4})$ is also 
contained in $\tilde B$, we introduce the following interpolating terms:

\begin{eqnarray}
&& \sum \pm
N([\cdots|(a_i',2h_i)|(a_i'',2g_j)|\cdots])\otimes
N([\cdots|(b_j, 2h_i+1)|\cdots]) , \label{inter_term1}\\
&& \sum  \pm \varepsilon(a_i''b_j)
N([(a_{i+1},2h_{i+1})|\cdots|(a_i',2h_i)|(b_{j+1},2g_{j+1})|\cdots|(b_{j-1},2g_{j-1})]),\label{inter_term2}  \\
&& \sum \pm N([\cdots|(a_i',2h_i)|(a_i'',2g_j+1)|\cdots])\otimes
N([\cdots|(b_j, 2g_j)|\cdots]) , \label{inter_term3}\\
&& \sum \pm\varepsilon(a_i'b_j)N([
(a_{i}'',2g_j+1)|\cdots|(a_{i-1},2h_{i-1})|(b_{j+1},2g_{j+1})|\cdots|(b_{j-1},2g_{j-1})]),\label{inter_term4} \\
&& \sum  \pm N([\cdots|(a_i,2g_j)|\cdots])\otimes
N([\cdots|(b_j', 2h_i)|(b_j'', 2g_j+1)|\cdots]) , \label{inter_term5}\\
&& \sum  \pm\varepsilon(a_i b_j')N([
(a_{i+1},2h_{i+1})|\cdots|(a_{i-1},2h_{i-1})|(b_{j}'',2g_j+1)|\cdots|(b_{j-1},2g_{j-1})]),\label{inter_term6} \\
&& \sum  \pm N([\cdots|(a_i,2h_i+1)|\cdots])\otimes
N([\cdots|(b_j', 2h_i)|(b_j'', 2g_j)|\cdots]) , \label{inter_term7}\\
&&\sum  \pm\varepsilon( a_ib_j'')N([
(a_{i+1},2h_{i+1})|\cdots|(a_{i-1},2h_{i-1})|(b_{j+1},2g_{j+1})|\cdots|
(b_{j}',2h_i)]) . \label{inter_term8}
\end{eqnarray}

One has
\begin{gather*}
(\ref{X_2})-(\ref{inter_term1})-(\ref{inter_term2}) \in\tilde B, \qquad
(\ref{X_4})-(\ref{inter_term5})-(\ref{inter_term6}) \in\tilde B ,\\
(\ref{inter_term3})-(\ref{X'_2})-(\ref{inter_term4}) \in\tilde B,\qquad
(\ref{inter_term7})-(\ref{X'_4})-(\ref{inter_term8}) \in\tilde B.
\end{gather*}
Moreover, 
\begin{gather*}
(\ref{inter_term1})=(\ref{inter_term3}),\quad
(\ref{inter_term5})=(\ref{inter_term7}),\quad
(\ref{inter_term2})=-(\ref{inter_term6}),\quad
(\ref{inter_term4})=-(\ref{inter_term8}).
\end{gather*}
Hence, 
$(\ref{X_2})+(\ref{X_4})-(\ref{X'_2})-(\ref{X'_4})\in \tilde B$.

By  a similar argument, the subspace spanned by 
elements of the form $X-X'-hX''$ in (\ref{equiv_rel2}) 
is also stable under $b$, and therefore $\tilde B$ 
is a subcomplex of $\tilde A$.
\end{proof}

\begin{theorem}\label{thm_quantization}Let $A=\tilde A/\tilde B$.
There is a DG Hopf algebra structure on $A$, which quantizes the DG
Lie bialgebra $(L, \{\,,\,\},\delta)$ of Theorem~\ref{thm_Liebialg}.
Moreover, $A$ is isomorphic to $U(L)[h]$ as $\k[h]$-modules.
\end{theorem}

The proof of the above theorem is given in the following 
subsections.

\subsection{Proof of DG algebra}

For any two elements $X, X'\in A$, define the product of $X$ and $X'$ as
follows: suppose  $X, X'$ are both represented by elements of the
form (\ref{typical_elmt}); let $X''$ be the element which is the
same as $X'$ but with the corresponding heights replaced by
$1+\max_{i,j,i',j'}(h_{i,j}(X)-h_{i',j'}(X'))$, 
where $h_{i,j}(X)$ are the heights in $X$ and similarly $h_{i',j'}(X')$ are the heights in $X'$.
Thus, $X''$ is obtained from
$X'$ by shifting the heights of the latter such that its heights are
larger than those of $X$. The product of $X$ and $X'$ is defined to be $X\bullet X''$. It is easy to see that this is well-defined, and commutes
with the boundary $b$.

\subsection{Proof of DG coalgebra}

For an element $X$ in the form of (\ref{typical_elmt}), let
$$P\, :=\, P_X \, :=\, \{(i,j)\mid 1\le i\le k, 1\le j\le p_i \} .$$
If $(i,j)\in P$, we let 
$$
(i,j)+(0,1) = \left\{\begin{array}{ll}
(i,j+1) & \mbox{ if } j<p_i, \\
(i,1)    & \mbox{ if } j=p_i,
\end{array}\right.
$$

Let $n$ be an integer greater than or equal to 2.
Now let $I$ be any subset of $P$ such that $\# I$ is even, and let $\phi: I\to I$ be an
involutive, fixed point-free map, where by being involutive we mean
$\phi^2=id$.
We call $(I,\phi,f)$ a
 \emph{$n$-labeling} of $X$ if
$$f:P\to\{1,2,\cdots, n\}$$
is a map such that:
\begin{equation}  \label{condi_1}
f(i,j)=\left\{\begin{array}{ll}f((i,j)+(0,1)),&\mbox{if}\,\, (i,j)\notin I;\\
f(\phi(i,j)+(0,1)),&\mbox{if}\,\, (i,j)\in
I,\end{array}\right.\end{equation} and
\begin{equation}\label{condition_2}
f(i,j)>f(\phi(i,j))\quad\mbox{if and only
if}\quad h_{i,j}>h_{\phi(i,j)},\,\,\mbox{for}\,\,(i,j)\in
I.\end{equation}

For an $n$-labeling $(I,\phi,f)$, let $q: P\to P$ be given by $$(i,j)\mapsto
\left\{\begin{array}{ll}(i,j)+(0,1),&\mbox{if\quad} (i,j)\notin
I,\\
\phi(i,j)+(0,1),&\mbox{otherwise,}\end{array}\right.$$ and define
$$g:P\backslash I\to P\backslash I$$
by the following: for $(i,j)\in P\backslash I$, let $g(i,j)$ be the first element not in $I$ under the iterations of the map $q$. Since $q$ is a permutation of the finite set $P$, $g$ is well-defined.

Suppose the orbits of $P$ under iterations of $q$ is
$\{Q_1,\cdots,Q_w\}$. Then $f$ descends to a map 
$\widehat f: \{Q_1,\cdots,Q_w\}\to\{1, \cdots, n\}$,
where $\widehat f(Q_m)= f(i,j)$ for any $(i,j)\in Q_m$, $1\le m\le w$.

Similarly, suppose the orbits of $P\setminus I$ under iterations of
 $g$ is $\{P_1,\cdots, P_l\}$. Then $f$ descends
to a map $\bar f:\{P_1, \cdots, P_l\}\to\{1, \cdots, n\}$,
where $\bar f(P_m)= f(i,j)$ for any $(i,j)\in P_m$, $1\le m\le l$.
Suppose $P_m$ ($1\le m\le l$)
is the orbit of $(i,j)$ under $g$; then we 
define an element $X_m\in LH$ by
$$X_m=N([(a_{i,j},h_{i,j})|(a_{g(i,j)},h_{g(i,j)})|\cdots]).$$

Let $1\le i\le n$.
Now define an element $X_{(I,\phi,f)}^{(i)}$  in $A$ 
by 
$$
X_{(I,\phi,f)}^{(i)} = \left\{ \begin{array}{ll}
1 & \mbox{ if } f^{-1}(i) =\emptyset ,\\
0 & \mbox{ if } \#(\bar f^{-1}(i))< \# (\widehat f^{-1}(i)), \\
X_{i_1}\bullet\cdots\bullet X_{i_r} & \mbox{ if }
\# (\bar f^{-1}(i))= \# (\widehat f^{-1}(i)) \mbox{ and }
\bar f^{-1}(i)=\{P_{i_1},\cdots, P_{i_r}\}.
\end{array}\right.
$$
The $n$-fold coproduct of $X$ is defined by
$$\Delta_{n} (X) \, :=\, 
\sum_{I, \phi, f}\varepsilon_{(I,\phi,f)} \, h^{(I, \phi,f)} \,
X_{(I,\phi,f)}^{(1)} \otimes \cdots\otimes 
X_{(I,\phi,f)}^{(n)},$$
where
$$
\varepsilon_{(I,\phi,f)}=\prod_{\{(i,j)\in I\, \mid\, f(i,j)<f(\phi(i,j)) \} }\varepsilon(a_{i,j}\cdot
a_{\phi(i,j)}) $$
and
$$
h^{(I,\phi,f)}=h^{ (\# I-2k+2l)/4}.
$$ 
Define $\Delta := \Delta_2$.
The following lemma yields the DG coalgebra on $A$:

\begin{lemma}Let $\Delta_n$ and $\Delta$ be defined as above.  Then

\vi  $\Delta_{n}$ is well-defined.

\vii $\Delta$ is coassociative.

\viii $b$ commutes with $\Delta$.
\end{lemma}

\begin{proof}
\vi To check that it is well-defined, 
one has to verify that if $\tilde
X\in\tilde B$, then $\Delta_{n-1}(\tilde X)\in
\displaystyle\sum_{i=1}^{n-1}id^{\otimes i-1}\otimes\tilde B\otimes
id^{\otimes n-i}$. The proof of this is completely similar to \cite[\S 3.5]{Schedler}
(by considering a quiver in \cite{Schedler} with just one vertex);  we omit the details.

\vii The proof is similar to \cite[\S 3.7]{Schedler}. We have
$$(\Delta\otimes id)\circ\Delta
=\Delta_3
=(id\otimes\Delta)\circ\Delta.$$
 As explained in \cite[\S 3.7]{Schedler}, one can group the
labelings $1$ and $2$ in $\Delta_3$ into labeling $1'$ and consider $1'$ and $3$;
this gives the first identity. Similarly, grouping the labelings $2$
and $3$ together into $2'$ and considering $1$ and $2'$ gives the second identity.

\viii The proof is by a direct verification similar to
the proof for the DG Lie coalgebra.
\end{proof}

\subsection{The Hopf identity} The proof is similar to \cite[\S 3.8]{Schedler}.
 For any $X, Y\in A$,
\begin{eqnarray*}\Delta(XY)&=&\Delta_2(XY)\\
&=&\sum_{\tiny\mbox{2-labelings\,of\,}XY}(XY)'\otimes
(XY)''\\
&=&\sum_{\tiny\mbox{2-labelings\,of\,} X\mathrm{\,and\,of\,}Y}
X'Y'\otimes
X''Y''  +\sum_{\phi(I\cap P_X)\cap P_Y\ne\emptyset}(XY)'\otimes
(XY)''.\end{eqnarray*}The last summation is over all 2-labelings $(I,\phi,f)$ of
$XY$ such that $\phi(I\cap P_X)\cap P_Y\ne\emptyset$.
 However, in the product $XY$, the heights of $Y$ are all greater than that of
$X$, and if the set $\phi(I\cap P_X)\cap P_Y$ is nonempty, then by
(\ref{condi_1}) and (\ref{condition_2}), one has
$$
0 = \sum_{(i,j)\in P_X}  f(i,j)-f\big( (i,j)+(0,1)\big)  
 =  \sum_{ \{  (i,j)\in I\cap P_X \,\mid\, \phi(i,j)\in P_Y \} } f(i,j)-f\big((i,j)+(0,1)\big) < 0,
$$
a contradiction. Hence, $\Delta(XY)= \Delta(X)\Delta(Y)$.

\begin{remark}
The antipode map $S:A\to A$ is defined by replacing the heights $h_{i,j}$ 
in $X$ by $-h_{i,j}$,
and then multiplying by $(-1)^{\#(P_X)}$.
 See  \cite[\S 3.9]{Schedler}.
\end{remark}

\subsection{Proof of the quantization}
Let $\mathrm B_C$ be a basis for $C$.
Let 
$$\mathrm B_L:= \{N([a_1|\cdots|a_n]) \in L
\mid a_i\in\mathrm B_C \mbox{ for all } i \}.$$
Then $\mathrm B_L$ is a basis for $L$. 
Let $SL$ be the symmetric algebra for $L$ and
$$\mathrm B_{SL}:= \{ x_1\bullet\cdots\bullet x_k \in SL \mid
x_i\in \mathrm B_L \mbox{ for all } i\}.$$
Then $\mathrm B_{SL}$ is a basis for $SL$.
Suppose $x\in \mathrm B_{SL}$ is the element
$$
 N([a_{1,1}|\cdots|a_{1,p_1}])\bullet\cdots\bullet
  N([a_{k,1}|\cdots|a_{k,p_k}])
$$
where $a_{i,j}\in\mathrm B_C$ for all $i,j$.
Then we fix an element $Y(x) \in \tilde{A}$ of the form
(\ref{typical_elmt}) where the sequence $h_{1,1},\ldots, h_{k,p_k}$ is a permutation of $1,2,\ldots, \#P_{Y(x)}$. 
Let $\bar Y(x) := Y(x)+\tilde B \in A$.

\begin{theorem}\label{PBW}
The set
$$ \mathrm B_A := \{\bar Y(x) \in A \mid x\in \mathrm B_{SL}\}$$
is a basis for $A$ over $\k[h]$.
\end{theorem}

We refer the reader to \cite[Corollary 4.2]{Schedler} for the proof of Theorem \ref{PBW}. (An alternative proof can also be given following the proof of the
usual PBW Theorem for universal enveloping algebras of Lie algebras given, for
example, in \cite[\S 9.2]{Carter}.)
It follows from Theorem \ref{PBW} and the PBW Theorem for $L$ that $A$ is isomorphic to 
$U(L)[h]$ as $\k[h]$-modules.

Note that any element of $L$ has a canonical lifting to an element of $A/hA$ since
by (\ref{equiv_rel2}) the heights do not matter.  Thus,  there is a natural map $L\to A/hA$.
By (\ref{equiv_rel1}), this map induces a homomorphism $\iota: U(L) \to A/hA$. 
It follows by Theorem \ref{PBW} that $\iota$ is an isomorphism of DG Hopf algebras.

Let $x=N([a_1|\cdots|a_n])\in L$,
and assume without loss of generality that its lifting 
$X\in A$ is represented by
$N([(a_1,1)|\cdots|(a_n,n)])$.
From the definition of $\Delta$, we have
\begin{eqnarray*}\Delta(X)&=&1\otimes X+X\otimes 1\\
&+&h\sum_{i<j}\pm\varepsilon(a_ia_j)N([(a_1,1)|\cdots|(a_{i-1},i-1)|\cdots|(a_{j+1},j+1)|\cdots|(a_n,n)])\\
&&\quad\quad\quad\quad\quad\quad\otimes N([(a_{i+1},i+1)|\cdots|(a_{j-1},j-1)])+\mbox{higher order terms}.
\end{eqnarray*}
It follows that
$$\delta(x)\equiv \frac{1}{h}(\Delta(X)-\Delta^{\mathrm{op}}(X))\mod
h.$$

\subsection{Proof of Theorem~\ref{main_thm}}

\vi If $(L,d)$ is a DG Lie algebra, then its universal enveloping
$U(L)$ with the induced differential is a DG Hopf algebra; denote
the induced differential by $b$. We have $H_*(U(L),b)$ is the
enveloping algebra of $H_*(L,d)$ (see Quillen~\cite[Appendix, Proposition 2.1]{Quillen}). Therefore, by Theorems
\ref{thm_Liebialg} and \ref{thm_quantization}, $H_*(A,b)$ quantizes
the Lie bialgebra $HC_*(C)[\mathsf m-1]$.

\vii This is immediate from Theorem \ref{thm_of_Jones} and 
Theorem \ref{main_thm}\vi.


\begin{thebibliography}{100}
\bibitem{ATZ}H. Abbaspour, T. Tradler and M. Zeinalian, {\it Algebraic string bracket as a Poisson bracket}, arxiv: 0807.2351, 2008.

\bibitem{AZ}H. Abbaspour and M. Zeinalian, {\it String bracket and flat connections},
Algebraic \& Geometric Topology 7 (2007) 197-231.

\bibitem{AMR1}J.E. Andersen, J. Mattes and N. Reshetikhin,
{\it The Poisson structure on the moduli space of flat connections and chord diagrams}.  Topology  35  (1996),  no. 4, 1069-1083.

\bibitem{AMR2}J.E. Andersen, J. Mattes and N. Reshetikhin,
{\it Quantization of the algebra of chord diagrams}.  Math. Proc. Cambridge Philos. Soc.  124  (1998),  no. 3, 451-467.

\bibitem{Carter} R. Carter, {\it Lie Algebras of Finite and Affine Type},
Cambridge Univ. Press, 2005.


\bibitem{CFP}A. Cattaneo, J. Fr\"ohlich and B. Pedrini,
{\it Topological field theory interpretation of string topology}. Comm. Math. Phys. 240 (2003), no. 3, 397-421.

\bibitem{CR}A. Cattaneo and C. Rossi, {\it Higher-dimensional $BF$ theories in the Batalin-Vilkovisky formalism: the BV action and generalized Wilson loops}.
Comm. Math. Phys. 221 (2001), no. 3, 591-657.

\bibitem{CS99}M. Chas and D. Sullivan, {\it String
topology}, arXiv:math-GT/9911159, 1999.


\bibitem{CS02}M. Chas and D. Sullivan, {\it Closed
string operators in topology leading to Lie bialgebras and higher
string algebra}, in {The legacy of Niels Henrik Abel}, 771-784,
Springer, Berlin, 2004.

\bibitem{Chen77}K.-T. Chen, {\it Iterated path integrals}, Bulletin
of the AMS, Volume 83, Number 5, September 1977, 831-879.


\bibitem{CV}R. Cohen and A. Voronov, {\it Notes on string topology},
in String Topology and Cyclic Homology, Adv. Courses in Math: CRM, Barcelona, Birkhauser (2006), 1-95. 

\bibitem{FT}Y. Felix and J.-C. Thomas, {\it Rational BV-algebra in String Topology},
 arXiv:0705.4194.

\bibitem{GJP}E. Getzler, J.D.S. Jones and S. Petrack,
{\it Differential forms on loop spaces and the cyclic bar complex.}
Topology 30 (1991), 339-371.

\bibitem{Ginzburg}V. Ginzburg, {\it Non-commutative symplectic geometry, quiver varieties, and operads}. Math. Res. Lett. 8 (2001), no. 3, 377-400.

\bibitem{Goldman}W. Goldman, {\it Invariant functions on Lie groups
and Hamiltonian flows of surface group representations}. {Invent.
Math.} 85 (1986), no. 2, 263-302.

\bibitem{Hamilton}A. Hamilton, {\it Noncommutative geometry and compactifications of the moduli space of curves},  arXiv: 0710.4603, 2007.

\bibitem{HL}A. Hamilton and A. Lazarev, {\it 
Symplectic $A_\infty$-algebras and string topology operations},
 arXiv:0707.4003.


\bibitem{HPS}K. Hess, P.-E. Parent and J. Scott, {\it CoHochschild homology of chain coalgebras}, J. Pure and Applied Algebra 213 no.4, (2009), 536-556,
arXiv:0711.1023. 

\bibitem{Jo87}J.D.S. Jones, {\it Cyclic homology and equivariant
homology}. Invent. Math. 87 (1987), 403-423.


\bibitem{LS07}P. Lambrechts and D. Stanley,
{\it Poincar\'e duality and commutative differential graded
algebras}. Annales Scientifiques de l'\'Ecole Normale Sup\'erieure
41 (2008) 495-509.

\bibitem{Loday}J.-S. Loday, {\it Cyclic homology}. Second edition.
Grundlehren der Mathematischen Wissenschaften, 301. Springer-Verlag, Berlin, 1998.



\bibitem{Quillen}D. Quillen, {\it Rational homotopy theory}. {Ann. of Math.} (2)
90, 1969, 205-295.

\bibitem{Schedler}T. Schedler, {\it A Hopf algebra quantizing a
necklace Lie algebra canonically associated to a quiver}. International Mathematics Research Notices, 2005, No. 12, 725-760.

\bibitem{Turaev}V.G. Turaev, {\it Skein quantization of Poisson algebras of
loops on surfaces}. {Ann. Sci. \'Ecole Norm. Sup.} (4) 24 (1991),
no. 6, 635-704.





\end{thebibliography}
\end{document}